\documentclass[12pt]{article}
\usepackage{amsmath,amsfonts,amssymb,amsthm}
\usepackage[a4paper,mag=1000,includefoot,left=2cm,right=2cm,top=2cm,bottom=2cm,headsep=1cm,footskip=1cm]{geometry}
\def\ve{\varepsilon}
\def\vk{\varkappa}
\def\leq{\leqslant}
\def\geq{\geqslant}
\def\half{{\textstyle\frac12}}
\def\onetwelve{{\textstyle\frac{1}{12}}}
\newtheorem{proposition}{Proposition}
\newtheorem{corollary}{Corollary}
\begin{document}
\noindent
{\Large A compact higher-order finite-difference scheme for the wave equation can be strongly non-dissipative on non-uniform meshes}

\medskip\par\noindent{{\large Alexander Zlotnik$^{a}$, Raimondas \v{C}iegis$^{b}$}

\par\noindent$^{a}$
{\small National Research University Higher School of Economics, 109028 Pokrovskii bd. 11, Moscow, Russia,
e-mail: \text{azlotnik@hse.ru}
\smallskip\par\noindent
$^b$ {\small Vilnius Gediminas Technical University,
Saul{\. e}tekio al. 11, LT-10223 Vilnius, Lithuania,\\
e-mail: \text{raimondas.ciegis@vgtu.lt}}}
\vskip -10pt
\renewcommand{\abstractname}{}
\begin{abstract}
\noindent We study necessary conditions for stability of a Numerov-type compact higher-order finite-diffe\-rence scheme for the 1D homogeneous wave equation in the case of non-uniform spatial meshes.
We first show that the uniform in time stability cannot be valid in any spatial norm provided that the complex eigenvalues appear in the associated mesh eigen\-value problem.
Moreover, we prove that then the solution norm grows exponentially in time making the scheme strongly non-dissipative and therefore impractical.
Numerical results confirm this conclusion.
In addition, for some sequences of refining spatial meshes, an excessively strong condition between steps in time and space is necessary (even for the non-uniform in time stability) which is familiar for explicit schemes in the parabolic
case.
\end{abstract}
\medskip\par\noindent
\textbf{1. Introduction}
\medskip\par
Vast literature is devoted to compact higher-order finite-difference schemes for various PDEs; we call only some recent papers for the wave equation \cite{BTT18,LLL19,STT19,ZKMMA2021}; see also references therein.
In the case of uniform meshes and PDEs with constant coefficients, their stability properties are often similar to more standard schemes and can be derived in a similar manner too.
But, for non-uniform spatial meshes, the stability study can become much more complicated and, more importantly, can lead to additional unpleasant conditions on the meshes involving ``strange'' conditions between steps in time and space.
For the time-dependent Schr\"{o}dinger equation describing another type of wave phenomena, both aspects have recently been presented in \cite{Z15,ZC18}.
Concerning the case of the non-uniform spatial meshes, see also \cite{ChS18,radziunas2014B}.

\par In this paper, we deal with a Numerov-type compact scheme for the 1D homogeneous wave equation.
Its uniform in time conditional stability and higher-order error properties even for non-smooth data have recently been demonstrated in \cite{ZKMMA2021} in the case of uniform meshes (and the non-homogeneous wave equation).
However, in the case of non-uniform spatial meshes, the situation can change dramatically.
This is owing to the fact that complex eigenvalues can appear
in the associated mesh eigenvalue problem in contrast to the differential case.

We first show that then the uniform in time stability
\textit{cannot be valid} in any spatial norm.
Moreover, we prove that they produce the exponential growth in time of the solution norm making the scheme \textit{strongly non-dissipative} and therefore impractical.
The included results of numerical experiments confirm this conclusion.
In addition, for some sequences of refining spatial meshes,
we obtain that the excessively strong stability condition
$\tau\leq c_0h_{\min}^2$ is \textit{necessary}
to prove even the non-uniform in time stability (similarly to the case of explicit schemes for \textit{the heat equation}).
Here $\tau$ is the time step and $h_{\min}$ is the minimal spatial one.
The possibility of such a non-standard phenomenon is interesting for
general theory of finite-difference schemes but, of course,
it is disappointing and indicates a non-robustness of such compact higher-order schemes.

\medskip\par\noindent\textbf{2. The compact higher-order scheme and necessary conditions for its stability}
\medskip\par
We deal with the initial-boundary value problem for the homogeneous 1D wave equation
\begin{gather}
 \partial_t^2u-a^2\partial_x^2u =0\ \ \text{in}\ \ \Omega\times (0,+\infty),
\label{hyperbeq1d}
\\[1mm]
 u|_{x=0,X}=0,\ \ u|_{t=0}=u_0(x),\ \ \partial_tu|_{t=0}=u_1(x),\ x\in\Omega:=(0,X),
\label{hyperbic1d}
\end{gather}
with $a>0$.
Recall that, for $t\geq 0$, the strong and weak energy relations hold
\begin{gather}
 \|\partial_tu(\cdot,t)\|_{L^2(\Omega)}^2+a^2\|\partial_xu(\cdot,t)\|_{L^2(\Omega)}^2
 =a^2\|\partial_xu_0\|_{L^2(\Omega)}^2+\|u_1\|_{L^2(\Omega)}^2,
\label{eneq1}
\\[1mm]
 \|u(\cdot,t)\|_{L^2(\Omega)}^2
 \leq\|u_0\|_{L^2(\Omega)}^2+a^{-2}\|u_1\|_{H^{-1}(\Omega)}^2,\ \ \text{with}\ \ H^{-1}(\Omega)=[H_0^1(\Omega)]^*.
\label{eneq2}
\end{gather}

\par We define a non-uniform mesh $\bar{\omega}_{h}$ on $\bar{\Omega}$, with the
nodes $0=x_0<\dots <x_N=X$ and the steps $h_j:=x_j-x_{j-1}$, together with the related well-known difference operator
\begin{gather*}
 \Lambda w=\tfrac{1}{\hat{h}}\big(\tfrac{w_+-w}{h_+}-\tfrac{w-w-}{h}\big)\ \
 \mbox{with}\ \
 \hat{h}=\tfrac{h+h_+}{2},\ \
  w_{\pm j}:=w_{j\pm 1},\ \ w_j=w(x_j).
\end{gather*}
Let $\omega_h:=\{x_i\}_{i=1}^{N-1}$.
We also recall the three-point Numerov-type averaging operator
\[
 s_Nw:=\tfrac{1}{12}(\alpha_h w_{-}+10\gamma_h w+\beta_h w_{+}),\  \text{with}\ \
 \alpha_h=2-\tfrac{h_{+}^2}{h\hat{h}},\
 \gamma_h=1+\tfrac{(h_{+}-h)^2}{5hh_{+}},\
 \beta_h=2-\tfrac{h^2}{h_{+}\hat{h}};
\]
for several its derivations and equivalent forms of $\alpha_h$, $\beta_h$ and $\gamma_h$, see \cite{JIS84,radziunas2014B,Z15,ZKMMA2021}.
One can check easily that $\onetwelve(\alpha_h+10\gamma_h+\beta_h)=1$ on $\omega_h$.
Recall that the properties $\alpha_{hj}\geqslant 0$ and $\beta_{hj}\geqslant 0$ mean the rather restrictive condition
$\tfrac{2}{\sqrt{5}+1}\leq\tfrac{h_{j+1}}{h_j}\leq\tfrac{\sqrt{5}+1}{2}$ (not in use below).
For the uniform mesh $\bar{\omega}_h$, clearly $\alpha_h=\beta_h=\gamma_h=1$ and one comes to the
well-known formula
$s_Nw=\onetwelve(w_{-}+10w+w_{+})$.
\par Let $H_h$ be the space of functions $w$: $\overline{\omega}_{h}\to\mathbb{C}$ such that $w_0=w_N=0$ and equipped with a mesh counterpart of the complex $L_2(\Omega)$--inner product $(v,w)_{\omega_h}:=\sum_{j=1}^{N-1}v_jw_j^*\hat{h}_j$ and the associated norm $\|w\|_{\omega_h}:=(w,w)_{\omega_h}^{1/2}$;
here $z^*$ is the conjugate for $z\in \mathbb{C}$.

\par We also define the uniform mesh in $t$ with the nodes $t_m=m\tau$, $m\geq 0$, and the step $\tau>0$, together with the related difference operators
\[
 \delta_ty:= \tfrac{\hat{y}-y}{\tau},\ \
 \Lambda_ty=\tfrac{\hat{y}-2y+\check{y}}{\tau^2},\ \
 \hat{y}^m:=y^{m+1},\ \
 \check{y}^m:=y^{m-1},\ \ \mbox{with}\ \  y^m=y(t_m).
\]

\par We study the following Numerov-type three-level compact scheme \cite{ZKMMA2021} for problem \eqref{hyperbeq1d}-\eqref{hyperbic1d}
\begin{gather}
  \big(s_N-\sigma\tau^2a^2\Lambda\big)\Lambda_tv^m-a^2\Lambda v^m=0\ \ \mbox{on}\ \ \omega_h,\ \ m\geq 1,
\label{eqhm}\\[1mm]
 v|_{j=0,N}=0,\ \
 \big(s_N-\sigma\tau^2a^2\Lambda\big)(\delta_tv)^0
 -\tfrac{\tau}{2}a^2\Lambda v_0
 =u_{1N}:=\big(s_N+\onetwelve\tau^2a^2\Lambda\big)u_1\ \ \mbox{on}\ \ \omega_h,
\label{eqh0}
\end{gather}
with $\sigma=\frac{1}{12}$,
where $v^0$ is given on $\overline{\omega}_h$.
Recall that its approximation orders are $O(\tau^4+h^4)$ and $O(\tau^4+h_{\max}^3)$ in the cases of the uniform mesh $\bar{\omega}_h$ with the step $h$ and general mesh $\bar{\omega}_h$ with $h_{\max}=\max_{1\leq j\leq N}h_j$, respectively.

\par For the uniform mesh $\bar{\omega}_h$ with the step $h=X/N$, under the condition
\begin{gather}
a^2\tfrac{\tau^2}{h^2}\leq 1-\ve_0^2\ \ \text{with any}\ \ 0<\ve_0<1,
\label{stabcond}
\end{gather}
the following uniform in time stability bounds (the mesh counterparts of \eqref{eneq1}-\eqref{eneq2}) hold \cite{ZKMMA2021}
\begin{gather}
 \ve_0^2\|\delta_tv^m\|_{s_N}^2+a^2\|s_tv^m\|_{-\Lambda}^2
 \leq a^2\|v^0\|_{-\Lambda}^2+\ve_0^{-2}\|u_{1N}\|_{s_N}^2,\ \ m\geq 0,
\label{stab1}\\[1mm]
 \ve_0\|v^m\|_{s_N}\leq \|v^0\|_{s_N}+2a^{-1}\|u_{1N}\|_{-\Lambda^{-1}},\ \ m\geq 0,
\label{stab2}
\end{gather}
for real $v_0$ and $u_{1N}$, where $\|w\|_A^2=(Aw,w)_{\omega_h}$ for any operator $A=A^*>0$ acting in $H_h$.
Recall that $\sqrt{\tfrac23}\|w\|_{\omega_h}\leq\|w\|_{s_N}\leq\|w\|_{\omega_h}$ for any $w\in H_h$.
It follows from the proof in \cite{ZKMMA2021} and \cite[Lemma 2.1]{Z94} that the left-hand side of \eqref{stab1} can be replaced by
$c(\ve_0)a^2(\|v^m\|_{-\Lambda}^2+\|v^{m+1}\|_{-\Lambda}^2)$ with some $c(\ve_0)>0$.
These bounds imply similar ones for the complex $v_0$ and $u_{1N}$ as well.

\par Our main aim is to derive necessary conditions for the validity of the stability bounds
\begin{gather}
 \|v^m\|_h\leq C(\|v^0\|_h+\|u_{1N}\|_{1h})\ \ \text{for any}\ \ v^0,u_{1N} \in H_h,\ \ m\geq 0,
\label{stabbound0}\\
 \|v^m\|_h\leq C
 (1+\vk\tau)^m(\|v^0\|_h+\|u_{1N}\|_{1h})\ \ \text{for any}\ \ v^0,u_{1N} \in H_h,\ \ m\geq 0,
\label{stabbound}
\end{gather}
with some $C>0$ and $\vk>0$ independent of the meshes, for $0<\tau\leq\tau_0$, expressing the stability with respect to the initial data.
Here $\|\cdot\|_h$ and $\|\cdot\|_{1h}$ are \textit{any} fixed norms in $H_h$.
The uniform in time bound \eqref{stabbound0} arises as a particular case of \eqref{stabbound} for $\vk=0$ (that we admit below).
Though bound \eqref{stabbound} is generally standard, it is too much broad compared to bounds \eqref{eneq1}-\eqref{eneq2} and \eqref{stab1}-\eqref{stab2} since it allows for exponential growth of approximate solutions for conservative problem \eqref{hyperbeq1d}-\eqref{hyperbic1d}.

\par To apply the spectral method,
we introduce the associated generalized mesh eigenvalue problem
\begin{gather}
 -\Lambda e=\lambda s_Ne,\ \ e\in H_h,
\ \ e\not\equiv 0.
\label{eigprob}
\end{gather}
For the uniform mesh $\bar{\omega}_h$ with the step $h$, we have $s_Nw=w+\tfrac{h^2}{12}\Lambda w$ and thus the eigenvalues $\lambda$ are positive real and are found explicitly; in particular, the maximal of them equals \cite{ZKMMA2021}                                                                                     \[
 \lambda_{\max}=\lambda_{\max}^{(-\Lambda)}/\big(1-\tfrac{h^2}{12}\lambda_{\max}^{(-\Lambda)}\big)
 <\tfrac{6}{h^2},\ \
 \lambda_{\max}=\tfrac{6}{h^2}\big(1+O(\tfrac{1}{N^2})\big),\ \ \text{with}\ \
 \lambda_{\max}^{(-\Lambda)}:=\tfrac{4}{h^2}\sin^2\tfrac{\pi(N-1)}{2N}.
\]

For the non-uniform mesh $\overline{\omega}_h$, we still
have $-\Lambda=-\Lambda^*>0$ in $H_h$ but, in general, $s_N$ loses the same properties (see \cite{Z15} for more details) that is crucial below.
Complex eigenvalues $\lambda=\lambda_R+\textrm{i}\lambda_I$,
$\lambda_I\neq 0$ (where $\textrm{i}$ is the imaginary unit), do exist for some $\bar{\omega}_h$, see Section 3 below.
\par Let $\{\lambda,e\}$ be an eigenpair of problem \eqref{eigprob}, $v^0=c_0e$ and $u_{1N}=c_1e$.
Then the function $v_j^m=e_jy^m$ solves equations \eqref{eqhm}-\eqref{eqh0} provided that $y$ satisfies the equations
\[
 \big(1+\onetwelve\tau^2a^2\lambda\big)\Lambda_ty^m+a^2\lambda y^m=0,\ \ m\geq 1,\ \
 \big(1+\onetwelve\tau^2a^2\lambda\big)(\delta_ty)^0+\tfrac{\tau}{2}a^2\lambda c_0=c_1,\ \
 y^0=c_0,
\]
or, equivalently,
\[
 \hat{y}-2\mu y+\check{y}=0,\ \
 y^1=\mu c_0+\tfrac{c_1}{1+\alpha},\ \ \text{with}\ \
 \mu=\tfrac{1-5\alpha}{1+\alpha},\ \ \alpha:=\onetwelve\tau^2a^2\lambda\neq -1.
\]

Considering the characteristic equation $q^2-2\mu q+1=0$ and its roots $q$ and $\frac{1}{q}$,
for $\mu\neq\pm 1$ (equivalently, $\alpha\neq 0,\half$), we can write the explicit formula
\begin{gather}
 y^m=a_+q^m+a_-q^{-m},\ \ m\geq 0,\ \ \text{with}\ \ a_{\pm}:=\tfrac12c_0\pm\tfrac{q}{q^2-1}\tfrac{c_1}{1+\alpha}.
\label{expl form}
\end{gather}

\par We can suppose that $|q|\geq 1$.
For this solution with some \textit{fixed} $c_0$ and $c_1$ such that $a_+\neq 0$ (for example, $c_0\neq 0$ and $c_1=0$), the stability bound \eqref{stabbound} as $m\to\infty$ implies \textit{the spectral necessary stability condition}
\begin{gather}
 |q|\leq 1+\vk\tau.
\label{spectrcond}
\end{gather}

\par The first our result is very simple but crucial from the general point of view.
\begin{proposition}
\label{prop1}
If $\lambda_I\neq 0$, then the uniform in $m\geq 0$ bound \eqref{stabbound0} cannot hold.
\end{proposition}
\begin{proof}
If $\vk=0$, then $|q|=1$ and $q^{-1}=q^*$, thus $\mu=\half(q+q^{-1})$ is real (and $|\mu|\leq 1$);
consequently $\alpha$ and $\lambda$ are real too.
\end{proof}

\newpage
For the uniform mesh $\bar{\omega}_h$, the inequality $|\mu|\leq 1$ is equivalent to $\alpha\leq\tfrac12$ for $\lambda>0$, thus we get the necessary condition
$\tfrac16a^2\tau^2\lambda_{\max}\leq 1$ for validity of bound \eqref{stabbound0}.
Consequently condition \eqref{stabcond} with $\ve_0=0$ is \textit{asymptotically sharp} as $h\to 0$.
One can check that bounds \eqref{stab1}-\eqref{stab2} are valid under the slightly sharper than \eqref{stabcond} condition $\tfrac16a^2\tau^2\lambda_{\max}\leq 1-\ve_0^2$ as well.

\par Now we obtain more detailed information than in Proposition \ref{prop1}.
\begin{proposition}
\label{prop2}
1. For $|q|\geq 1$, the following asymptotic formula holds
\begin{gather}
 |q|=1+\vk_0\tau+O(\tau^2|\lambda|)\ \ \text{as}\ \ \tau^2|\lambda|\to 0,\ \ \text{with}\ \ \vk_0=\vk_0(\lambda)\equiv \tfrac{a}{\sqrt{2}}\sqrt{|\lambda|-\lambda_R}.
\label{asymp formula}
\end{gather}

\smallskip\par 2. For validity of the stability bound \eqref{stabbound}, the following conditions are necessary:
\begin{gather}
 \tfrac{1}{6}a^2\tau^2\max\{|\lambda_R|,2|\lambda_I|\}
 \leq 1+O(\tau),\ \
 a^2|\lambda_I|\tau\leq \tfrac92\vk+O(\tau)\ \ \text{as}\ \ \tau\to 0.
\label{necess cond}
\end{gather}
\end{proposition}
\begin{proof}
1. If $|\alpha|=\tfrac{1}{12}\tau^2a^2|\lambda|\to 0$, then clearly
\[
 \mu=1-6\alpha+O(|\alpha|^2),\ \
 \mu^2-1=-12\alpha+O(|\alpha|^2).
\]
Since $\lambda=|\lambda|e^{\textrm{i}\gamma}$ with $\gamma=\arg\lambda$, we get
\[
 \mu+\sqrt{\mu^2-1}=1\pm\tau a|\lambda|^{1/2}e^{\textrm{i}(\gamma+\pi)/2}+O(\tau^2\lambda),
\]
therefore
\[
 |\mu+\sqrt{\mu^2-1}|^2=1\pm 2\tau a|\lambda|^{1/2}\cos\tfrac{\gamma+\pi}{2}+O(\tau^2\lambda).
\]
For $|q|\geq 1$, we thus derive
\[
 |q|=1+\tau a|\lambda|^{1/2}|\sin\tfrac{\gamma}{2}|+O(\tau^2\lambda).
\]
Since $|\sin\tfrac{\gamma}{2}|=[\tfrac{1-\cos\gamma}{2}]^{1/2}$, we obtain formula \eqref{asymp formula}.

\par 2. We first suppose that $|\mu|^2\leq 1+12\ve$, with some $0<\ve<2$, i.e.
\[
 (1-5\alpha_R)^2+25\alpha_I^2\leq (1+12\ve)[(1+\alpha_R)^2+\alpha_I^2].
\]
This means that
\begin{gather}
 \alpha_I^2\leq\tilde{\ve}+\alpha_R\Big(2a_{0\ve}-\alpha_R\Big),
 \ \ \text{with}\ \ \tilde{\ve}:=\tfrac{\ve}{2-\ve},\ \ a_{0\ve}:=\tfrac{1+2\ve}{2(2-\ve)}=\tfrac14+O(\ve)\ \ \text{as}\ \ \ve\to 0.
\label{ineq alp}
\end{gather}
Consequently
\begin{gather}
\max\,\{|\alpha_R-a_{0\ve}|,|\alpha_I|\}\leq a_{0\ve}+\sqrt{\tilde{\ve}}
\label{ineq alpRI}
\end{gather}
and from \eqref{ineq alp} we obtain
\[
 |1+\alpha|^2=(1+\alpha_R)^2+\alpha_I^2\leq b_\ve:=1+\tilde{\ve}+2(1+a_{0\ve})(2a_{0\ve}+\sqrt{\tilde{\ve}})=\tfrac94+O(\sqrt{\ve}).
\]
\par Next, writing $q=\rho e^{\textrm{i}\varphi}$ with $\rho\geq 1$, we get
\[
 \mu=\mu_R+\textrm{i}\mu_I=\tfrac12(q+q^{-1})
 =\tfrac12(\rho+\rho^{-1})\cos\varphi+\textrm{i}\tfrac12(\rho-\rho^{-1})\sin\varphi
\]
and therefore
\[
 |\mu|^2\leq 1+\big[\tfrac12(\rho-\rho^{-1})\big]^2,\ \
 |\mu_I|=\tfrac{6|\alpha_I|}{|1+\alpha|^2}\leq\tfrac12(\rho-\rho^{-1}).
\]
Inequality \eqref{spectrcond} implies that
\[
 \tfrac12(\rho-\rho^{-1})\leq\vk_\tau\tau\ \ \text{with}\ \ \vk_\tau:=\vk\tfrac{1+0.5\vk\tau}{1+\vk\tau}=\vk(1+O(\vk\tau))\ \ \text{as}\ \ \tau\to 0.
\]
Thus one can take $\ve=\ve_\tau:=\onetwelve\vk_\tau^2\tau^2$ for $\vk_\tau^2\tau^2<24$ at the beginning of the proof and derive
\begin{gather}
 6|\alpha_I|=\tfrac12\tau^2a^2|\lambda_I|\leq b_{\ve_\tau}\vk_\tau\tau.
\label{mu neccond}
\end{gather}
Now inequalities \eqref{ineq alpRI} and the last one imply \eqref{necess cond}.
\end{proof}
\begin{corollary}
\label{cor1}
Let $T>0$ be any fixed number and $\tau=\frac{T}{M}$, $M\geq 1$.
Let also $y^0=c_0e$ and $u_{1N}=0$ for clarity.
Then the following asymptotic formula holds
\begin{gather}
 \|y^M\|_h=\|y^0\|_h\big(1+\vk_0\tau+O(\tau^2|\lambda|)\big)^M
 =\|y^0\|_h\big(e^{\vk_0T}+O(\tau T|\lambda|)\big)\ \ \text{as}\ \ \tau^2|\lambda|\to 0.
\label{normyM}
\end{gather}
\end{corollary}
\par Note that the first necessary condition \eqref{necess cond} \textit{for any} $\vk$ is in the spirit of the above mentioned condition
$\tfrac16a^2\tau^2\lambda_{\max}\leq 1$ for the uniform mesh and $\vk=0$.
Moreover, in general for $\vk=0$ (and thus $\lambda_I=0$), one can omit $O(\tau)$ there and also get that $\lambda_R\geq 0$ (taking $\ve=0$ in the proof).
\par We emphasize that if $\lambda_I\neq 0$, then the behavior of $\|y^M\|_h$ in \eqref{normyM} for \textit{any} $\tau$ small enough is in a deep contrast with bounds \eqref{eneq1}-\eqref{eneq2} and \eqref{stab1}-\eqref{stab2} since
$e^{\vk_0T}\gg1$ even for moderate values of $\vk_0T$.
This means that, under all these assumptions, scheme \eqref{eqhm}-\eqref{eqh0} is \textit{strongly non-dissipative} and can hardly be used in practice.

\par Let $\overline{\omega}_{h}^{\,0}$ be a fixed mesh with the nodes $0=x_0^0<\ldots<x_{N_0}^0=X$ (with $N_0\geq 2$) and the mean step $h_{\omega^0}=X/N_0$.
We recall the family of meshes $\overline{\omega}_{h}^{\,0,K}$, $K\geq 1$, on $\bar{\Omega}$ with the nodes
\begin{gather*}
 x_{2kN_0+l}=2k\tfrac{X}{K}+\tfrac{x_l^0}{K},\ 0\leq 2k<K,\ \
 x_{(2k-1)N_0+l}=2k\tfrac{X}{K}-\tfrac{x_{N_0-l}^0}{K},\ 1\leq 2k-1<K,\ \ x_N=X
\end{gather*}
for any $0\leq l\leq N_0-1$,
with $N=N_0K$ and the mean step $h_\omega=X/N=h_{\omega^0}/K$,
together with the related extension operator $\Pi_K$: $H(\overline{\omega}_h^{\,0})\to H(\overline{\omega}_h^{\,0,K})$ such that
\begin{gather*}
 \Pi_Kw_{2kN_0+l}=w_l,\ 0\leq2k<K,\ \
 \Pi_Kw_{(2k-1)N_0+l}=-w_{N_0-l},\ 1\leq 2k-1<K,
\end{gather*}
for any $0\leq l\leq N_0-1$, see \cite{ZC18}.
The next result has also been proved in \cite{ZC18}.
\begin{proposition}
\label{prop3}
Let $\{\lambda^{(0)},e\}$ be an eigenpair of problem \eqref{eigprob} for $\overline{\omega}_h=\overline{\omega}_h^{\,0}$.
Then \linebreak
$\{\lambda^{(0)}K^2,\Pi_Ke\}$ is an eigenpair of problem \eqref{eigprob} for $\overline{\omega}_h=\overline{\omega}_h^{\,0,K}$.
\end{proposition}
\begin{corollary}
For the meshes $\overline{\omega}_h=\overline{\omega}_h^{\,0,K}$, the necessary condition \eqref{necess cond} takes the form
\begin{equation} \label{eq:1}
 a^2|\lambda_I^{(0)}|h_{\omega^0}^2\tfrac{\tau}{h_\omega^2} \leq \tfrac92\vk+O(\tau)\ \ \text{as}\ \ \tau\to 0.
\end{equation}
Notice that here $\tau$ stands on the left in contrast to $\tau^2$ in \eqref{stabcond} in the case of the uniform mesh.
\end{corollary}
Let us briefly discuss the case of general $\sigma$ in scheme \eqref{eqhm}-\eqref{eqh0}.
Then the explicit formula \eqref{expl form} remains valid with $\mu=(1+(\sigma-0.5)\tau^2a^2\lambda)/(1+\alpha)$
and $\alpha=\sigma\tau^2a^2\lambda$.
One can check easily that Proposition \ref{prop1}, Proposition \ref{prop2}, Item 1 and thus Corollary \ref{cor1} remain valid independently of $\sigma$.
Proposition \ref{prop2}, Item 2 can be extended for $\sigma<\frac14$ but we do not come into details.
Instead, notice only that now $|\mu_I|=\tfrac{a^2\tau^2|\lambda_I|}{2|1+\alpha|^2}$ and thus,
under the qualitatively natural assumption that $\tau^2a^2|\lambda|\leq C_0$,
we have $|1+\alpha|\leq 1+|\sigma|C_0$ and
in the proof of Proposition \ref{prop2}, Item 2 inequality \eqref{mu neccond} is replaced easily by the new necessary condition
\[
 a^2\tau|\lambda_I|\leq2(1+|\sigma|C_0)^2\vk_\tau
\]
with the same $\vk_\tau$
that is in the spirit of the second condition \eqref{necess cond}.

\medskip\par\noindent
\textbf{3. Numerical experiments}
\medskip\par

First, applying the brute force search strategy,
the following critical mesh $\overline{\omega}_{h}^{\,0}$ on $[0,1]$, with $N_0=14$ nodes and the space mesh steps
$(h_1,\ldots,h_{14})=\frac{1}{51}(2,2,1,4,2,1,3,3,6,5,6,5,6,5)$
has been found (see also \cite{ZC18}).
The corresponding maximal in modulus $\lambda=3529.9\pm\textrm{i}27.2044$ for problem \eqref{eigprob} are complex whereas the other eigenvalues are positive real.
Then the above defined family of meshes $\overline{\omega}_{h}^{\,0,K}$ is applied with $N=KN_0$,
for several values of $K$ and time segments $[0,T]$.
Note that then $\vk_0\approx 0.2289K$ in \eqref{asymp formula} and \eqref{normyM}.

We take $a=1$ and the initial data
\[
u_0(x) = \exp\big[-\big(10(x-0.5)\big)^4\big],\ \
u_1(x) = 0,  \quad 0 \leq x \leq 1.
\]
Let $v_i^0=u_0(x_i)$.
Define the error $e_K(T):=\max_{0\leq j\leq KN_0}|u(x_j,T)-v_j^M|$, where $\tau M=T$.
Its catastrophic \textit{growth} for $\tau = 0.01/K$ as $K$ increases is seen from the following numerical results for $T=2$ and 4:
\begin{align*}
& e_{20}(2) \approx 1.443\cdot 10^{-4},\ \; e_{40}(2) \approx 3.867\cdot 10^{-2},\ \;
  e_{60}(2) \approx 64.497,\ \;             e_{80}(2) \approx 1.082\cdot 10^5,\\
& e_{10}(4) \approx 3.648\cdot 10^{-3},\ \; e_{20}(4) \approx 1.345,\ \;
  e_{30}(4) \approx 1926,\ \;               e_{40}(4) \approx 3.281\cdot 10^6.
\end{align*}
The results for smaller values of $\tau$ are even worse.

In addition, $e_{20}(4)\approx 39.9225$ and $e_{20}(6)\approx 609643$ for $M=14400$.
Consequently the practical rate of the exponential growth of the error in time can be estimated as $\vk_{pr}:=\half\ln(e_{20}(6)/(e_{20}(4))\approx 4.817$ that is in a good agreement with its theoretical estimate $\vk_0\approx 4.579$ for $K=20$.

These numerical results confirm the above theoretical ones that the rate of non-dissipativity of
the compact higher-order scheme rapidly grows as $T$ or $K$ increases.

\medskip\par\noindent
\textbf{Acknowledgements}
\medskip\par
The work of the first author has been accomplished within the framework of the Academic Fund Program at the
National Research University Higher School of Economics (HSE) in 2019--2020 (grant no. 19-01-021)
and the Russian Academic Excellence Project ``5-100''
as well as the Russian Foundation for the Basic Research, grant no.~19-01-00262.
\renewcommand{\refname}{\large\textbf{References}}
\small{
}
\end{document}